\documentclass[review]{elsarticle}
\usepackage{lineno}
\usepackage[latin1]{inputenc}
\usepackage{times}
\usepackage{amssymb,mathrsfs, amsmath}
\usepackage{amsmath,amsthm}
\usepackage{amssymb}
\usepackage{latexsym}
\usepackage{amsfonts}
\usepackage{mathrsfs}
\usepackage{epsfig}
\usepackage{hyperref}
\usepackage{graphicx}
\usepackage{graphics}

\newtheorem{thm}{Theorem}[section]

\newtheorem{defn}{Definition}[section]
\newtheorem{prop}{Proposition}[section]

\newtheorem{lem}{Lemma}[section]

\newtheorem{rem}{Remark}[section]

\newtheorem{cor}{Corollary}[section]

\newtheorem{exmpls}{Examples}[section]

\journal{Indian Journal of Mathematics}

\begin{document}

\begin{frontmatter}

\title{Conformal Riemannian morphisms between Riemannian manifolds}

\author{R.B. Yadav \fnref{myfootnote}\corref{mycorrespondingauthor}}
\address{Sikkim University, Gangtok, Sikkim, 737102, \textsc{India}}
\ead{rbyadav15@gmail.com}
\author{Srikanth Venkatesh Kuppum}
\address{Indian Institute of Technology Guwahati, Assam 781039, \textsc{India}}
\ead{kvsrikanth@iitg.ernet.in}
\begin{abstract}
 In this article we introduce conformal Riemannian morphisms. The idea of conformal Riemannian morphism generalizes the notions of isometric immersions,  Riemannian submersions and   Riemannian maps.

\end{abstract}

\begin{keyword}
\texttt{Isometric Immersion, Riemannian Submersion, Riemannian map, Conformal Riemannian map}
\MSC[2010] 53C20\sep 53C42 \sep  53B20 \sep 58A05 \sep 57R35
\end{keyword}

\end{frontmatter}



\section{Introduction}
A generalization of the notions of isometric immersion and Riemannian submersion between Riemannian manifolds was introduced by Arthur E. Fischer in 1992 and was named as Riemannian map \cite{Fischer}.  The theory of isometric immersions originated from Gauss's studies on surfaces in  Euclidean spaces. On the other hand, the study of Riemannian submersions between Riemannian manifolds was initiated by O'Neil \cite{O'Neill} and Gray \cite{Gray}. Later on,  Sahin gave a generalization of  Riemannian maps and defined Conformal Riemannian maps \cite{Sahin}. The immersions and submersions are generalized and unified by subimmersions. Fischer introduced Riemannian map as a Riemannian analog  of subimmersion.\\
In this paper first we introduce  conformal Riemannian morphisms and give some examples. Next, we give some necessary and sufficient conditions for a smooth function  between Riemannian manifolds to be conformal Riemannian morphism, Theorem \ref{rmnp12}, \ref{rmnp13}. Also, we  prove that a conformal Riemannian morphism on a compact Riemannian manifold is subimmersion (Definition \ref{rmnd1}),  Proposition  \ref{rmnp16}. Thus conformal Riemannian morphism can be seen as   yet another Riemannian   analog of the notion of subimmersion .  Next, we prove that subimmersions are equivalent to generalized conformal maps (Definition \ref{rb557}) if domain has sufficiently large dimension, Theorem \ref{rb558}. Finally,  we give  local structure of conformal Riemannian morphisms, Theorem \ref{rmn30}.\\
Throughout this paper $M$ and $ N$ denote Riemannian manifolds  with Riemannian metrics $g_M$ and $g_N$, respectively.
From \cite{Fischer}, \cite{Sahin} and \cite{Baired} we recall following definitions.
  \begin{defn}\label{rmn6}
   A smooth function $f:M\to N$   is called \textbf{ conformal Riemannian map}   if there exist a linear subspace $H_x\subset T_xM$, a number $\wedge_f(x)\in \mathbb{R}$ and $rank(d_xf)\ne 0$, for each   $x\in M$  such that
  \begin{enumerate}
    \item $H_x=\ker(df_x)^{\perp}$, for each $x\in M,$
    \item  $g_N(f(x))(df_xX,df_xY)=\wedge_f(x) g_M(x)(X,Y),$ for every  $X,Y \in H_x$, for each $x\in M.$
  \end{enumerate}
  From definition, it is clear that $\wedge_f(x)>0$ and the function $\wedge_f:M\to \mathbb{R}^{+}$ is smooth. Thus  a smooth function $f:M\to N$ with $d_xf>0$, $\forall$ $x\in M$  is \textbf{ conformal Riemannian map}   if there exist a linear subspace $H_x\subset T_xM$, for each   $x\in M$ and a  smooth function $\wedge_f:M\to \mathbb{R}^{+}$ such that
  \begin{enumerate}
    \item $H_x=\ker(df_x)^{\perp}$, for each $x\in M,$
    \item  $g_N(f(x))(df_xX,df_xY)=\wedge_f(x) g_M(x)(X,Y),$ for every  $X,Y \in H_x$, for each $x\in M.$
  \end{enumerate}
If $\wedge_f(x)=1,$ for all $x\in M,$ then a conformal Riemannian map is called a \textbf{ Riemannian map}. If $\ker(df_x)=\{0\}$, for each $x\in M,$ then a conformal Riemannian map is called a \textbf{conformal immersion}. If $\wedge_f(x)=1,$ for each $x\in M,$ then a conformal immersion  $f$ is said to be \textbf{ isometric immersion}.  If $df_x$ is surjective, for each x in M, then a conformal Riemannian map is called a \textbf{ horizontally conformal submersion}. If $df_x$ is surjective, for each x in M, then a  Riemannian map is called a \textbf{ Riemannian submersion}. $f$ is called an \textbf{ isometry}  if it is isometric immersion as well as Riemannian submersion.
\end{defn}
\section{Geometric functions and Conformal Riemannian morphisms}

\begin{defn} \label{def:ipm} Let V and W be inner--product spaces over the field of real numbers and $f:V\rightarrow W$ be a linear transformation. Denote the kernel and range of $f$ by $K$ and $R(f)$, respectively. We say that $f$ is a  \textbf{geometric function}  if there
exists a subspace $C$ of $V$ such that $V=K\oplus C$,  and
$f_{|C} :C\rightarrow R(f)\mathrm{\; is \;  \; conformal,}$ that is, for some $r>0$, $\langle f(u),f(v)\rangle =r\langle u,v\rangle$, for every $ u,v\in C .$
 We call $C$ a \textbf{Conf subspace} corresponding to  $f$ and denote it  by $\mathrm{Conf}(f)$. We call r \textbf{conformality factor} associated to $f$. We denote the collection of all geometric functions from an inner product space V to another inner product space W by $Geom(V,W).$
\end{defn}
We now give two general results on the existence of geometric functions.
\begin{prop}\label{rb555}
 Let V, W be inner product spaces over $\mathbb{R}$ and $T:V\to W$ be a linear map of rank k. If $\dim V\ge 2k-1$, then T is geometric.
\end{prop}
\begin{proof}
 Let $dimV=n\ge 2k-1$ and $T^*$ be adjoint of T. Since $T^*T$ is self adjoint, there exists an orthonormal basis $\{u_i\}_{1}^n$ of V such that   $T^*Tu_i=\lambda_iu_i$  for  $\lambda_i\in\mathbb{R}$, for all $i=1,2,\cdots, k$ and $T^*Tu_{k+1}=\cdots =T^*Tu_n=0$.
 Since $0\le \langle  Tu_i,Tu_i\rangle =\langle u_i,T^*Tu_i\rangle=\langle u_i,\lambda_iu_i\rangle=\lambda_i\|u_i\|^2$
  and $\mathrm{rank}( T)=k$,  $\lambda_i> 0,$ for all $1\le i\le k.$ We can reorder $\{u_i\}_{1}^n$ in such a way that $0<\lambda_1<\lambda_i$, for all $2\le i\le k.$ Now define $\alpha_i=\sqrt{\frac{\lambda_1}{\lambda_i}}$, $E_1=u_1$, $E_i=\alpha_i u_i+\sqrt{1-\alpha_i^2}u_{k+i-1}$, for all $2\le i\le k. $ Clearly $\{E_i\}_{1}^k$ is orthonormal subset of V and $\|TE_i\|=\lambda_1$, for all $1\le i\le k.$  So, T is geometric with $\mathrm{conf}(T)=\mathrm{span} \{E_i\}_{1}^k$ and conformality factor $\sqrt{\lambda_1}.$
\end{proof}
\begin{prop}\label{rb556}
  Let V, W be inner product spaces over $\mathbb{R}$,  $T:V\to W$ be a linear map of rank k and  $\dim V\ge 2k$. Then T is geometric.  Moreover, there exist uncountably many $\mathrm{conf}(T)$ subspaces such that  conformality factors associated to different $\mathrm{conf}(T)$ subspaces are different.
\end{prop}
\begin{proof}
   Let $dimV=n\ge 2k$ and $T^*$ be adjoint of T. As in the proof of Proposition \ref{rb555},  there exists an orthonormal basis $\{u_i\}_{1}^n$ of V such that $T^*Tu_i=\lambda_iu_i$  for  $0<\lambda_i\in\mathbb{R}$, for all $i=1,2,\cdots, k$ and $T^*Tu_{k+1}=\cdots =T^*Tu_n=0$.
  We can reorder $\{u_i\}_{1}^n$ in such a way that $0<\lambda_1<\lambda_i$, for all $2\le i\le k.$  Choose  $0<\alpha<\lambda_1.$  Now define $\alpha_i=\sqrt{\frac{\alpha}{\lambda_i}}$,  $E_i=\alpha_i u_i+\sqrt{1-\alpha_i^2}u_{k+i}$, for all $1\le i\le k. $ Clearly $\{E_i\}_{1}^k$ is orthonormal subset of V and $\|TE_i\|=\alpha$, for all $1\le i\le k.$  So, T is geometric with $\mathrm{conf}(T)=\mathrm{span} \{E_i\}_{1}^k$ and conformality factor $\sqrt{\alpha}.$ Clearly for different values of $\alpha$ in the interval $(0,\lambda_1)$, $\mathrm{conf}(T)$ is different.  This completes the proof.
\end{proof}

Next we define conformal Riemannian morphisms and conformal Riemannian submersions. 
\begin{defn}\label{rmn1}
 A smooth function $f:M\to N$   is called  \textbf{conformal Riemannian morphism}   if  there exist a linear subspace $H_x\subset T_xM$, for each   $x\in M$ and a  smooth function $\wedge_f:M\to \mathbb{R}^{+}$ such that
  \begin{enumerate}
    \item $H_x\oplus \ker (df_x)=T_xM$, for each $x\in M;$ and
    \item  $g_N(f(x))(df_xX,df_xY)=\wedge_f(x) g_M(x)(X,Y),$ for every  $X,Y \in H_x$, for each $x\in M.$\\

 \end{enumerate}

  Thus a smooth function $f:M\to N$   is called  \textbf{conformal Riemannian morphism}   if there exists a  smooth function $\wedge_f:M\to \mathbb{R}^{+}$ such that $df_x$ is a geometric function  with conformality factor $\wedge_f(x),$ for all $x\in M.$
   The  function $\wedge_f$ is called   \textbf{conformality  factor}  associated to $ f$.\\ If $\wedge_f(x)=1,$ for each  $x\in M$, we call $f$ a \textbf{Riemannian morphism}.

\end{defn}
\begin{defn}\label{rmn2}
A smooth function $f:M\to N$  is  called a \textbf{ conformal Riemannian submersion}  if  there exist a linear subspace $H_x\subset T_xM$, for each   $x\in M$ and a  smooth function $\wedge_f:M\to \mathbb{R}^{+}$ such that
  \begin{enumerate}
  \item $df_x$ is surjective, for each x in M;
    \item $H_x\oplus \ker (df_x)=T_xM$, for each $x\in M;$ and
    \item  $g_N(f(x))(df_xX,df_xY)=\wedge_f(x) g_M(x)(X,Y),$ for every  $X,Y \in H_x$, for each $x\in M.$\\

 \end{enumerate}
 Clearly, every conformal Riemannian  submersion is a conformal Riemannian  morphism.
\end{defn}
Next, we give some examples of conformal Riemannian morphisms. We let $\mathbb{R}^n$ denote the Euclidean n-space taken with its standard  flat metric.
\begin{exmpls}
\begin{enumerate}
\item   If $F:M\to N$ is a  conformal Riemannian map, then F is a conformal Riemannian morphism with  $\ker(dF_x)^{\perp}=H_x=Conf(df_x),$ $\forall x\in M_1.$
\item Define $f:\mathbb{R}^4\to\mathbb{R}^4$ by $f(x_1,x_2,x_3,x_4)=(e^{x_3}(x_1-x_2),0,0,e^{x_3}(x_4-x_2)),$   for all $(x_1,x_2,x_3,x_4)\in \mathbb{R}^4.$ Then $f$ is a smooth function and  for any $a\in \mathbb{R}^4,$ we have,
      $$df_{(a_1,a_2,a_3,a_4)}=\left( \begin{array}{cccc}
                                  e^{a_3} & - e^{a_3} &  e^{a_3}(a_1-a_2) & 0 \\
                                  0 & 0 & 0 & 0 \\
                                  0 & 0 & 0 & 0 \\
                                  0 & -e^{a_3} & e^{a_3}(a_4-a_2) &  e^{a_3}
                                \end{array}\right).$$
                                Note that $df_a(e_1)=e^{a_3}e_1$, $df_a(e_4)=e^{a_3}e_4$, $df_a(e_1+e_4+e_2)=0$, $\;\;\;\;\;df_a(a_2-a_1,0,1,a_2-a_4)=0$.  Let $K=\mathrm{span}\{(1,1,0,1), (a_2-a_1, 0, 1, a_2-a_4)\}$ and $T=\mathrm{span}\{e_1,e_4\}.$  Then $(df_a)_{|K}=0$ and $f$ is conformal Riemannian morphisms with $\mathrm{Conf}(df_a)=T$ and $\wedge_f(a)=e^{2a_3}$, for all  $a\in \mathbb{R}^4$. Note that $\mathrm{Conf}(df_a)$ and $\ker(df_a)$ are not orthogonal in this example. So $f$ is not a conformal Riemannian map.

\item From  Proposition \ref{rmnp1} it is clear that every  smooth function $f:\mathbb{R}^n\to \mathbb{R}$ is a conformal Riemannian morphism if and only if it is a constant function or has nowhere zero  gradient.
\item\label{rmn10} Every smooth curve $c:\mathbb{R}\to (M,g_M)$ is conformal Riemannian morphism if and only if  it is a constant curve or $c^{\prime}(t)\ne 0$, for all $t\in \mathbb{R}$. Note that if c is a curve which is a  conformal Riemannian morphism, then $\wedge_c(t)=g_M(t)(c^{\prime}(t),c^{\prime}(t)).$
\item A linear isomorphism $A:\mathbb{R}^n\to\mathbb{R}^n$  is a conformal Riemannian morphism  if and only if $A$ is a scalar multiple of an orthogonal transformation.

\item  A conformal Riemannian morphism need not be a harmonic map,  and a harmonic map need not be conformal Riemannian morphism. For example, the curve $c:\mathbb{R}^{+} \to \mathbb{R}^2 $ given by $c(t)=(t^3,t^6)$ is conformal Riemannian morphism but not harmonic. Also, $f:\mathbb{R}^2\to \mathbb{R}^2$ defined by $f(x,y)=(2x,3y)$ is harmonic map but not conformal Riemannian morphism.
\item From \cite{Fuglede}, we   have following characterization of harmonic morphisms: a smooth map $\phi :M\to N$  between Riemannian  manifolds is a harmonic morphism if and only if it is both harmonic and a (horizontally) conformal submersion away from points where $d\phi=0$. So we conclude that a harmonic morphism is conformal Riemannian morphism  away from points where $d\phi=0$ but converse need not be true because a conformal Riemannian morphism may not be a harmonic map.
     \item Composition of conformal Riemannian morphisms need not be a conformal Riemannian morphism. For example, $f:\mathbb{R}\to\mathbb{R}^2$ given by $f(x)=(x^2+x,x^2)$, and $g:\mathbb{R}^2\to\mathbb{R}$ given by $g(x,y)=x$ are conformal Riemannian morphisms but $g\circ f(x)=x^2+x$ is not conformal Riemannian morphism, from Example \ref{rmn10}, because $(g\circ f)^{\prime}(-\frac{1}{2})=0$ and  $g\circ f$ is not constant.
\end{enumerate}
\end{exmpls}
\section{Necessary and sufficient conditions for a map to be  conformal Riemannian morphism }\label{sec:prm}
In this section we give some necessary and sufficient conditions for a map between Riemannian manifolds to be conformal Riemannian morphism.
\begin{prop}\label{rmnp1}
Let $f:M\to \mathbb{R}$ be a smooth function on a connected Riemannian manifold $M$. Then $f$ is a conformal Riemannian morphism if and only if $f$ is constant or it has nowhere zero gradient.
\end{prop}
\begin{proof}
 If $f$ is constant,  then it is easy to see that $f$ is a conformal Riemannian morphism, take $H_x=0$, $\wedge_f(x)=1$, $\forall x\in M$. Let $\mathrm{grad}(f)$ denote the  gradient of  $f$ which is a vector field defined by $df_x (X)=g_M(x)(\mathrm{grad}( f)(x),X)$, for each  $X\in T_xM$, and   $K_x=\{\mathrm{span}(\mathrm{grad}(f))\}^{\perp},$ for all $x\in M$. We have,  $df_x(\lambda\; \mathrm{grad} (f)(x))=g_M(x)(\mathrm{grad} (f)(x),\lambda\; \mathrm{grad} (f)(x))=\lambda \;\|\mathrm{grad} (f)(x)\|^2$ and  $df_x{|K_x}\equiv 0.$ Now\\
 $ \|\mathrm{grad}( f)(x)\|^2g_M(x)(\lambda_1\;\mathrm{grad} ( f)(x),\lambda_2\; \mathrm{grad}( f)(x))$
 \begin{eqnarray*}
   &=& \lambda_1\;\lambda_2\;\|\mathrm{grad}( f)(x)\|^4 \\
    &=& \lambda_1\;\|\mathrm{grad} (f)(x)\|^2 .\lambda_2\;\|\mathrm{grad} (f)(x)\|^2\\
    &=& df_x(\lambda_1 \;\mathrm{grad} (f)(x)).df_x(\lambda_2\; \mathrm{grad} (f)(x))
 \end{eqnarray*}

   From above discussion, we conclude that $f$ is a conformal Riemannian morphism with $\mathrm{Conf}(df_x)=\mathrm{span}(\mathrm{grad} (f)(x))$, $\wedge_f(x)=\|(\mathrm{grad} f)(x)\|^2 $ $\iff$ $f$ is constant or $\mathrm{grad}(f)(x)\ne 0, \;\forall x\in M.$
\end{proof}

\begin{cor}\label{rmnp2}
   Let $f:M\to \mathbb{R}$ be a smooth function. If $\|\mathrm{grad}( f)(x)\|=1$, for each $x\in M$,  then $f$ is a  Riemannian map.
\end{cor}
\begin{proof}
 From the  proof of  the Proposition \ref{rmnp1}, we have $\mathrm{Conf}(df_x)=\mathrm{span}(\mathrm{grad} (f)(x))$ and $Ker (df_x)=\{\mathrm{span}(\mathrm{grad}(f))\}^{\perp}$, $\wedge_f(x)=\|\mathrm{grad}( f)(x)\|^2=1$. So $f$ is a  Riemannian map.
\end{proof}

\begin{defn}\label{rmn4}
    Let  $f:M\to N$ be a smooth function. Let $H_x\subset T_xM$ be a linear subspace,   for each   $x\in M$ such that  $H_x\oplus \ker (df_x)=T_xM$.
    \begin{itemize}
      \item Define, $\forall x\in M$, a linear  function $({df}_{H_x})^{\diamond}:T_{f(x)}N\to T_xM$ by
$$ ({df}_{H_x})^{\diamond}(X) = \begin{cases}
      ((df_x)_{|H_x})^*(X), & \mbox{if } X\in \mathrm{range}(df_x) \\
      0, & \mbox{if } X\in \mathrm{range}(df_x)^{\perp}. \end{cases}$$ Here $((df_x)_{|H_x})^*$ is adjoint of $((df_x)_{|H_x}):H_x\to \mathrm{range}(df_x).$
      \item Define,  $\forall x\in M$,  $P_{H_x}:T_xM\to T_xM$ by $P_{H_x}(x)=({df}_{H_x})^{\diamond}\circ (df_x).$
      \item Define,  $\forall x\in M$,  $Q_{H_x}:T_{f(x)}N\to T_{f(x)}N$ by $Q_{H_x}(x)=(df_x)\circ ({df}_{H_x})^{\diamond}.$
    \end{itemize}

   \end{defn}

\begin{thm}\label{rmnp13}
  Let  $f:M\to N$  be a smooth function. Then $f$ is conformal Riemannian morphism if and only if  there exist a smooth function $\wedge_f:M\to \mathbb{R}^{+}$ and a subspace $H_x\subset T_xM$ with $T_xM=H_x\oplus \ker{df_x}$ such that  $P_{H_x}\circ P_{H_x}=\wedge_f(x)P_{H_x},$ for all $x\in M$.
\end{thm}

\begin{proof}
$f$ is conformal Riemannian morphism if and only if there exists a smooth function $\wedge_f:M\to \mathbb{R}^{+}$ and a subspace $H_x\subset T_xM$ with $T_xM=H_x\oplus \ker{df_x}$ such that $\wedge_f(x)g_M(x)(X_1,Y_1)=g_N(f(x))(df_xX_1,df_xY_1),$ for all $X_1,Y_1\in H_x$. Since, for all $X_1\in H_x,$  $Y_1\in T_xM,$ $g_N(f(x))(df_xX_1,df_xY_1)=g_M(x)(X_1,((df_x)_{|H_x})^*\circ df_xY_1)=g_M(X_1,(df_{H_x})^{\diamond}\circ df_xY_1),$  $f$ is conformal Riemannian morphism if and only if there exists a smooth function $\wedge_f:M\to \mathbb{R}^{+}$ and a subspace $H_x\subset T_xM$ with $T_xM=H_x\oplus \ker{df_x}$ such that
\begin{equation}\label{eqn1}
  \wedge_f(x)g_M(x)(X_1,(df_{H_x})^{\diamond}\circ df_xY_1)=g_N(f(x))(df_xX_1,df_x\circ(df_{H_x})^{\diamond}\circ df_xY_1),
\end{equation}
for all $X_1\in H_x, Y_1\in T_xM$.
Equation \ref{eqn1} is equivalent to
\begin{equation}\label{eqn2}
\wedge_f(x)g_M(X_1,(df_{H_x})^{\diamond}\circ df_xY_1)=g_M(x)(X_1,(df_{H_x})^{\diamond}\circ df_x\circ (df_{H_x})^{\diamond}\circ df_xY_1),
\end{equation}
 for all $X_1\in H_x,$  $Y_1\in T_xM.$
Since $P_{H_x}= (df_{H_x})^{\diamond}\circ(df_x),$  Equation  \ref{eqn2} is equivalent to
\begin{equation}\label{eqn3}
  \wedge_f(x)g_M(x)(X_1,P_{H_x}Y_1)=g_M(x)(X_1,P_{H_x}\circ P_{H_x}Y_1),
\end{equation}
 for all $X_1\in H_x,$  $Y_1\in T_xM.$ Since $range(P_{H_x})=H_x$,  Equation  \ref{eqn3} is equivalent to  $P_{H_x}\circ P_{H_x}=\wedge(x)P_{H_x}$. Hence we conclude the result.

\end{proof}
\begin{thm}\label{rmnp12}
   Let  $f:M\to N$  be a smooth function. Then $f$ is conformal Riemannian morphism if and only if  there exist a smooth function $\wedge_f:M\to \mathbb{R}^{+}$ and a subspace $H_x\subset T_xM$ with $T_xM=H_x\oplus \ker{df_x}$ such that  $Q_{H_x}\circ Q_{H_x}=\wedge_f(x)Q_{H_x},$ for all $x\in M$.
\end{thm}

\begin{proof}
  $f$ is conformal Riemannian morphism if and only if there exists a smooth function $\wedge_f:M\to \mathbb{R}^{+}$ and a subspace $H_x\subset T_xM$ with $T_xM=H_x\oplus \ker{df_x}$ such that $\wedge_f(x)g_M(x)(X_1,Y_1)=g_N(f(x))(df_xX_1,df_xY_1),$ for all $X_1,Y_1\in H_x$. Since $range(({df}_{H_x})^{\diamond})=H_x,$ $\wedge_f(x)g_M(x)(X_1,Y_1)=g_N(f(x))(df_xX_1,df_xY_1),$ for all $X_1,Y_1\in H_x$, if and only if, for all $X_2\in range(df_x)$, $Y_2\in T_{f(x)}N$ \begin{equation}\label{eqn4}
  \scriptsize{  \wedge_f(x)g_M(x)((df_x)_{|H_x})^*X_2,(df_{H_x})^{\diamond}Y_2)=g_N(f(x))(df_x\circ((df_x)_{|H_x})^*X_2,df_x\circ(df_{H_x})^{\diamond}Y_2).}
   \end{equation} Equation \ref{eqn4} is equivalent  to
    \begin{equation}\label{eqn5}
    \scriptsize{
       \wedge_f(x)g_N(f(x))(X_2,df_x\circ(df_{H_x})^{\diamond}Y_2)=g_M(x)((df_x)_{|H_x})^*X_2,((df_x)_{|H_x})^*\circ df_x\circ(df_{H_x})^{\diamond}Y_2),}
    \end{equation}  for all $X_2\in range(df_x)$, $Y_2\in T_{f(x)}N$. Equation \ref{eqn5} is equivalent to
  \begin{equation}\label{eqn6}
   \scriptsize{ \wedge_f(x)g_N(f(x))(X_2,df_x\circ(df_{H_x})^{\diamond}Y_2)=g_N(f(x))(X_2,((df_x)_{|H_x})\circ (df_{H_x})^{\diamond}\circ df_x\circ(df_x)^{\diamond}Y_2),}
  \end{equation}
   for all $X_2\in range(df_x)$, $Y_2\in T_{f(x)}N$. Equation \ref{eqn6} is equivalent to
 \begin{equation}\label{eqn7}
  \scriptsize{ \wedge_f(x)g_N(f(x))(X_2,df_x\circ(df_{H_x})^{\diamond}Y_2)=g_N(f(x))(X_2,df_x\circ (df_{H_x})^{\diamond}df_x\circ(df_{H_x})^{\diamond}Y_2),}
 \end{equation}
  for all $X_2\in range(df_x)$, $Y_2\in T_{f(x)}N$.

  We have $Q_{H_x}=(df_x)\circ (df_{H_x})^{\diamond}$.   So, Equation \ref{eqn7} holds for all  $X_2\in range(df_x)$, $Y_2\in T_{f(x)}N$, if and only if  $Q_{H_x}\circ Q_{H_x}=\wedge(x)Q_{H_x}$, for each $x\in M$. Hence we conclude the result.

\end{proof}
\begin{prop}\label{rmnp11}
  Let  $f:M\to N$ be a smooth function.  If  $H_x=(\ker (df_x))^{\perp}$, then $ (df_{H_x})^{\diamond} =(df_x)^{*}.$
\end{prop}
\begin{proof}

  Suppose that $H_x=(\ker (df_x))^{\perp}$. Then, for all $X\in T_xM$ and $Y\in T_{f(x)}N$, we have $X=X_1+X_2$ and $Y=Y_1+Y_2$ for unique $X_1\in H_x$, $X_2\in (\ker (df_x)),$ $Y_1\in range(df_x)$ and $Y_2\in (range(df_x))^{\perp}.$  We have, $$g_N(f(x))((df_x)(X),Y)=g_M(X,(df_x)^*(Y)),$$ for all $X\in T_xM$, $Y\in T_{f(x)}N$, implies that $range((df_x)^*)=(\ker (df_x))^{\perp}$ and  $range(df_x)=(\ker ((df_x)^*))^{\perp},$ for all $x\in M$.  So, we have $$g_N(f(x))((df_x)(X),Y)=g_M(X,(df_x)^*(Y)),$$ for all $X\in T_xM$, $Y\in T_{f(x)}N$, if and only if  $$g_N(f(x))((df_x)(X_1),Y_1)=g_M(X_1,(df_x)^*(Y_1)),$$  for all $X_1\in H_x$, $Y_1\in range(df_x).$ So, $(df_{H_x})^{\diamond}_{|range(df_x)}=((df_x)^{\ast})_{|range(df_x)}$. Since $(df_{H_x})^{\diamond}_{|(range(df_x))^{\perp}}\equiv 0$, $((df_x)^{\ast})_{|(range(df_x))^{\perp}}\equiv 0$, this implies that $ (df_{H_x})^{\diamond} =(df_x)^{*}$.
\end{proof}
\begin{rem}\label{rmnrem1}
  \begin{enumerate}
    \item From Proposition \ref{rmnp11}, we conclude that $ (df_{H_x})^{\diamond}$ is a generalization of $(df_x)^{*}.$
    \item From Theorem \ref{rmnp13},  $f$ is Riemannian morphism if and only if there exists a subspace $H_x\subset T_xM$ with $T_xM=H_x\oplus \ker{df_x}$ such that  $P_{H_x}$ is a projection operator, that is, $P_{H_x}\circ P_{H_x}=P_{H_x}$.
      \item From Theorem \ref{rmnp13}, $f$ is Riemannian morphism if and only if there exists a subspace $H_x\subset T_xM$ with $T_xM=H_x\oplus \ker{df_x}$ such that  $Q_{H_x}$ is a projection operator, that is, $Q_{H_x}\circ Q_{H_x}=Q_{H_x}$.
  \end{enumerate}
\end{rem}

\section{Subimmersions and generalized-conformal maps}

\begin{defn}\label{rmnd1}
  A smooth function $f:M\to N$ is a subimmersion at $x\in M$, if there is an open set U containing x, a manifold P, a submersion $S:U\to P,$ and an immersion $J:P\to N$ such that $f{|U}=J\circ S.$  A smooth function $f:M\to N$ is a subimmersion if it is a subimmersion at each $x\in M.$

\end{defn}
 Now we state a proposition from \cite{Fischer} which relates subimmersions and maps of locally constant rank.

\begin{prop}\label{rmnp15}
   A smooth function $f:M\to N$ is a subimmersion if and only if the rank function $x\mapsto rank(df_x)$ ( equivalently, if and only if the nullity function $x\mapsto \dim(\ker(df_x))$) is locally constant, and hence constant on the connected components of M.
 \end{prop}

\begin{proof}
  If $f$ is a subimmersion, then it is locally composition of a submersion and an immersion  ,  and hence  has locally constant rank.  Conversely, if it has locally constant then, by using rank theorem, we conclude that it is a subimmersion. Local constancy of nullity follows from rank nullity theorem.
\end{proof}
We state following theorem about subimmersion \cite{Abraham}, \cite{Ddne1}).
\begin{thm}
  Let $f:M\to N$ be a subimmersion.
  \begin{enumerate}
    \item Then for every $y\in f(M)\subset N$, $f^{-1}(y)$ is a smooth closed submanifold of M, with  $T_x(f^{-1}(y))=\ker(df_x),$ for every $x\in f^{-1}(y).$
    \item For every $x\in M,$ there exists a neighborhood U of x such that $f(U)$ is a submanifold of N with $T_{f(x)}(f(U))=range(df_x).$ If, in addition, $f$ is open or closed   onto its image, then $f(M)$ is a submanifold of N.
    \item  If $df_x:T_xM\to T_{f(x)}N$ is not surjective, then U in (2)  above may be chosen so that  $f(U)$ is nowhere dense in N.
  \end{enumerate}
\end{thm}
\begin{thm}\label{rmnt20}
   Let $f:M\to N$ be a subimmersion. Then
   \begin{enumerate}
     \item If $f$ is injective, then $f$ is an immersion, and hence an injective immersion.
     \item If $f$ is surjective and M is connected, then $f$ is a submersion, and hence a surjective submersion.
     \item More generally, if M is connected and $f$ is open or closed onto its image, then $f(M)$ is a submanifold of N, and the range restricted map $f:M\to f(M)$ is a  surjective submersion onto $f(M).$
         \item If $f$ is bijective and M is connected , then $f$ is a diffeomorphism.
   \end{enumerate}
\end{thm}
We have following lemma from the thesis of first author \cite{rbthesis}.
\begin{lem}\label{lem:closednopen}
Let $m$ and $n$ be two natural numbers. Further, let $M(m\times n,\mathbb{R})$ denote the inner-product space  of all $m\times n$ matrices with real entries. Let $r$ be any  integer. Define $$O_r	=\{A|A\in M(m\times n,\mathbb{R}),\; \mathrm{rank}(A)\geqslant r\}\; \text{and} $$
$$C_r	=\{A|A\in M(m\times n,\mathbb{R}),\; \mathrm{rank}(A)\leqslant r\}$$. Then,   for each $r$, $O_r$  and $C_r$ are respectively
  open and closed in $M(m\times n,\mathbb{R})$.
\end{lem}
\begin{proof}
Since $C_r=O_{r+1}^\complement$, it suffices to show that $O_r$ is open  in $M(m\times n,\mathbb{R})$ for each integer $r$.
We dispose off the special cases when $r\leqslant 0$ or $r > \min\{m,n\}$. If $r \leqslant 0$, then $O_r=M(m\times n,\mathbb{R})$ which is  open  in $M(m\times n,\mathbb{R})$. On the other hand if $r > \min\{m,n\}$, then $O_r$ is the empty set  which too is open in $M(m\times n,\mathbb{R})$.

Next, suppose $0 < r \leqslant \min\{m,n\}$. In any $m\times n$ matrix the number of minors of size $r$
is given by $\binom{m}{r} \binom{n}{r}$. Denote this number by $k$.  Let $p_i:M(m\times n,\mathbb{R})\rightarrow M(r\times r,
\mathbb{R})$  for $ i \in \{1,2,\ldots, k\}$ be the projection maps corresponding to these $k$ minors. Let $\delta:M(r\times r,\mathbb{R})\rightarrow \mathbb{R}$ be the determinant function. Since $p_i$ and $ \delta$ are continuous functions, the composition $\delta\circ p_i$ is a continuous function from $M(m\times n,\mathbb{R})$ to $\mathbb{R}$  for each $i \in \{1,2,\ldots,k\}$. Consequently, for each $i\in\{1,2,\ldots,k\}$, the set $U_i= (\delta\circ p_i)^{-1}(\mathbb{R}\backslash\{0\})$ is open. Now, every matrix in $M(m\times n,\mathbb{R})$ of rank at least $r$ belongs to at least one of $U_i$ for $i\in\{1,2,\ldots,k\}$. Thus, $O_r=\bigcup_{i=1}^{k}U_i$ is  open in $M(m\times n,\mathbb{R})$.
\end{proof}
\noindent Lemma \ref{lem:closednopen} leads to the following theorem.
\begin{thm}\label{thm:closednopen}
Let $V$ and $W$ be finite--dimensional real inner--product spaces.
Let r be any integer. Define
$ O_r=\{f\,|\,f:V\rightarrow W,\; f \mathrm{\;is\;linear},\;\mathrm{rank}(f)\geqslant r\} $ and
$C_r=\{ f\,|\,f:V\rightarrow W,\; f \mathrm{\;is\;linear},\;      \mathrm{rank}(f) \leqslant r\}.$ Then, for each $r$,   $O_r$  and $C_r$ are respectively open and closed in Hom($V,W$).
\end{thm}
We have following  lemma

\begin{lem}\label{krsn2}
Let V and W be finite--dimensional real inner--product spaces and let $T:V\rightarrow W$ be a linear transformation. Suppose $\{T_k\}$ is a sequence in Geom ($V,W$) which converges to $T$. Suppose that, $\forall k$,   $r_k= \mathrm{conformality\; factor\; of}\; T_k$ and $0<p\le r_k$, for some real number p.  Then, terms of the sequence $\{T_k\}$ eventually have the same rank as $T$.
\end{lem}

\begin{proof}
Let the dimension of $V$ be $n$ and rank($T$) be $r$. For each natural number $k$, let $X_k$ and $Y_k$ denote, respectively, the kernel of $T_k$ and Conf subspace corresponding to  $T_k$.

 Since   $T_k\rightarrow T$, there exists a natural number $K_1$ such that for any natural number $k> K_1$, we have $\|T_k-T\| < \frac{\sqrt{p}}{2}$. We claim that for each natural number  $k>K_1$, the intersection $\mathrm{ker}(T)\cap Y_k=\{0\}$.  To prove this, it suffices to show that if
$\alpha\in\mathrm{ker}(T)$ with $\|\alpha\|=1$, then $\alpha\not\in Y_k$, $\forall  k>K_1$. For $\alpha\in\mathrm{ker}(T)$ with $\|\alpha\|=1$,  we have ,  $\forall  k>K_1$, $\|T_k(\alpha)\|=\|T_k(\alpha)-T(\alpha)\|\leqslant\|T_k-T\|\|\alpha\| < \frac{\sqrt{p}}{2} $. Also, $\alpha\in Y_k$ implies that  $\|T_k(\alpha)\|\ge \sqrt{p}$.  So, we conclude that $\alpha\not\in Y_k,\; \forall k>K_1$.\\
Fix a natural number $k>K_1$ and  suppose that $\{u_i\}_{i=1}^{n-r}$ is a basis of ker($T$). Since each $T_k$ is a geometric function, we have $V=X_k\oplus Y_k$. Thus, we can find $v_i\in X_k$ and $w_i\in Y_k$ such
that $u_i=v_i+w_i$ for all $i\in\{1,2,\ldots,n-r\}$. We claim that $\{v_i\}_{i=1}^{n-r}$ is linearly independent. Suppose that for some scalars $\beta_i$, we have $\sum_{i=1}^{n-r} \beta_i v_i =0$. Since each $v_i= u_i-w_i$, we get $\sum_{i=1}^{n-r} \beta_i u_i =\sum_{i=1}^{n-r} \beta_i w_i$ which implies that $\sum_{i=1}^{n-r} \beta_i u_i \in Y_k$. From our previous observation that $\mathrm{ker}(T)\cap Y_k=\{0\}$, we have $\sum_{i=1}^{n-r} \beta_i u_i =0$ which by linear independence implies that $\beta_i=0$ for all $i\in\{1,2,\ldots, n-r\}$. Conclude that $\{v_i\}_{i=1}^{n-r}$ is linearly independent and hence rank($T_k)\leqslant$ rank($T$).  Since, from Theorem \ref{thm:closednopen}, $\{S:rank(T)\le rank(S)\}$ is open, there exists $\epsilon>0$ such that $B_{\epsilon}(T)\subset \{S:rank(T)\le rank(S)\}$. Since $\{T_n\}$ is eventually in $B_{\epsilon}(T)$, for some natural number $K_2 $,  $k> K_2\Rightarrow rank(T_k)\ge rank(T)$. Hence  there exists a natural number $K=\max\{K_1,K_2\}$ such that for all $k>K$, we have rank($T_k)=$ rank($T$).
 \end{proof}
 \begin{prop}\label{rmnp16}
Let $f:M\rightarrow N$ be a conformal Riemannian  morphism. If the conformality factor $\wedge_f$ is bounded below by a positive number, then  f is a subimmersion. In particular, if M is compact, then  f is a subimmersion.
\end{prop}
\begin{proof}
Let the dimensions of $M$ and $N$ be $m$ and $n$ respectively and $\wedge_f:M\to \mathbb{R}$ be a smooth function such that such that $c<\wedge_f(x), \;\forall x\in M,$ for some number $c> 0.$  Define $\rho:M \rightarrow \mathbb{Z}$  by $x\mapsto \mathrm{rank}(df_x)$. We shall  show that $\rho$ is a locally constant function on $M$.

On  the contrary suppose that $\rho$ is not locally constant at a $p\in M$.  Choose  co--ordinate charts $(U,\phi)$   and $(V,\psi)$ around p and $f(p)$ respectively. Let $g$ denote the function $\psi\circ f\circ\phi^{-1}:\phi(U)\rightarrow \psi(V)$.  Since $f$ is a conformal Riemannian morphism, for each $y\in \phi(U)$ we may view $dg_y:\mathbb{R}^m\rightarrow\mathbb{R}^n$ as a geometric function between real inner--product spaces. And since $\phi$ and $\psi$ are diffeomorphisms,   the function  $y\mapsto \mathrm{rank}(dg_{y})$ is not locally constant at $q=\phi(p)$. Hence, there exists a sequence $\{q_k\}$ in $\phi(U)$ such that $q_k\rightarrow q$ but  has infinitely many terms     $q_l$ with          $\mathrm{rank}(dg_{q_l})\ne \mathrm{rank}(dg_q)$, i.e. the sequence $\mathrm{rank}(d g_{q_k})$ is not eventually constant. This contradicts Lemma \ref{krsn2}  and makes untenable our assumption that $\rho$ is not locally constant at $p$.  If M is compact then the conformality factor $\wedge_f$ being a continuous real valued function on a compact space is bounded below by a positive number.

\end{proof}

\begin{defn}\label{rb557}
    We define a smooth function $f:M\to N$   to be   \textbf{generalized conformal map}   if there exists a  positive real number $\lambda_f$ such that $df_x$ is a geometric function  with conformality factor $\lambda_f,$ for all $x\in M.$ 
Clearly, a generalized conformal map $f:M\to N$  is a conformal Riemannian morphism with a constant conformality factor $\wedge_f(x)=\lambda_f$, $\forall x\in M$.  Converse is not true because conformality factor of f, $\wedge_f(x)$ may not be a constant function.
\end{defn}
\begin{thm}\label{rb558}
  Let $f:M\to N$ be a smooth function,  M be  compact and $dim M \ge 2 \text{rank}(df_x),$ for all $x\in M$ . Then f is    subimmersion if and only if it is  a generalized conformal map.
\end{thm}
\begin{proof} Let f be a subimmersion.
Without any loss of generality we assume that M is connected. Hence $\text{rank}(df_x)=k$, for some fixed integer k, $\forall x\in M$. This implies $\text{rank}(df_x^*df_x)=k$, $\forall x\in M$.  By Wielandt-Hoffman inequality  for Hermitian matrices  \cite{Hf-Wt} singular value function  $S$  defined on the set of $n\times n$ real symmetric matrices by $S(A)=(\lambda_1(A), \cdots, \lambda_n(A))$, where $\lambda_1(A)\le \lambda_2(A)\le \cdots\le  \lambda_n$ are eigen values of A,  is continuous. Since   $x\to df_x^*df_x$ is smooth and S is continuous,  $x\to S df_x^*df_x$ is continuous. We define $E_f:M\to \mathbb{R}$ by $E_f(x)=\text{ smallest  positive eigen value of} (df_x)*(df_x),$ $\forall x\in M.$ Since  $df_x^*df_x$  has all its eigen values nonnegative,$S (df_x^*df_x)=(0,0,\cdots, 0, \lambda_1(df_x^*df_x), \cdots, \lambda_k(df_x^*df_x))$. Clearly  $E_f(x)=\lambda_1(df_x^*df_x)$. So $E_f$ is continuous. So, since M is compact, $E_f$ is bounded below by a positive number $\beta_f$. So $0<\beta_f\le \lambda_1(df_x^*df_x)$, $\forall x\in M.$  As in the proof of Proposition \ref{rb556} if we choose $0<\alpha < \beta_f\le \lambda_1(df_x^*df_x)$ then $df_x$ is geometric with conformality factor $\sqrt{\alpha}$, $\forall x\in M$.  Converse part follows from  the Proposition \ref{rmnp16}.
\end{proof}
We state following proposition from \cite{Fischer}
\begin{prop}\label{rmnp18}
  A map $f:M_1\to N$  from a manifold $M_1$ to a Riemannian manifold $N$ is an immersion iff the pullback tensor $f^*g_N$ is a Riemannian metric on M. In this case, $f:(M_1,f^*g_N)\to (N,g_N)$  is an isometric immersion.
\end{prop}
Now we show that locally a conformal Riemannian morphism is composition of a conformal Riemannian submersion and isometric immersion. This gives information about local structure of conformal Riemannian morphisms.
\begin{prop}\label{rmnp19}
Let $f:M\rightarrow N$ be a conformal Riemannian  morphism  such that conformality factor $\wedge_f$ is bounded below by a positive number. Then $f$ is composition of a conformal Riemannian submersion and isometric immersion.

\end{prop}
\begin{proof}
From Proposition \ref{rmnp16}, f is a subimmersion. For $x\in M,$ let U, P, S and J be as in Definition \ref{rmnd1} for a subimmersion, so that $f_U=J\circ S.$ Let $g_U=(g_M){|U}$ denote the restriction  of the Riemannian metric to the open set U of M, and let $g_P=J^*g_N,$ be the pullback tensor on P. We show that  $(U,g_U)$ and $(P,g_P)$ are Riemannian manifolds, the submersion $S:(U,g_U)\to (P,g_P)$ is a conformal Riemannian submersion, and the immersion $J:(P,g_P)\to (N,g_N)$ is an isometric immersion.
Since U is open in M, $(U,g_U)$ is an open Riemannian submanifold of $(M,g_M)$. Since J is an immersion by Proposition \ref{rmnp18} $g_P$ is a Riemannian metric on P, and $J:(P,g_P)\to (N,g_N)$ is an isometric immersion. Let $\wedge$ be the conformality factor associated to $f$.  Since $f$ is a conformal Riemannian morphism, $g_N(f(x))(df_xX,df_xY)=\wedge(x)g_M(x)(X,Y),$ for all $X,Y\in \mathrm{Conf}(df_x).$ Now  $g_N(f(x))(df_xX,df_xY)=\wedge(x)g_M(x)(X,Y),$ for all $X,Y\in \mathrm{Conf}(df_x),$ if and only if $\wedge(x)g_M(x)(X,Y)=g_N(f(x))(dJ_{S(x)}\circ dS_xX,dJ_{S(x)}\circ dS_xY),$ for all $X,Y\in \mathrm{Conf}(df_x), $ if and only if $\wedge(x)g_M(x)(X,Y)=g_P(S(x))(dS_xX, dS_xY),$ for all $X,Y\in \mathrm{Conf}(df_x). $ Hence S is a conformal Riemannian submersion.
\end{proof}

\begin{thm}\label{rmn30}
  Let $f:M\rightarrow N$ be a conformal Riemannian  morphism such that conformality factor $\wedge_f$ is bounded below by a positive number. Then
  \begin{enumerate}
    \item If $f$ is injective, then $f$ is an injective conformal immersion.
    \item If $f$ is surjective and M is connected, then  $f$ is a surjective  conformal Riemannian submersion
    \item More generally,  if $f$ is open or closed onto its image, then $f(M)$ is submanifold of N. Let $g^N_{f(M)}$ denote  the Riemannian metric induced on $f(M)$ by the metric $g_N$. Then $f:(M,g_M)\to (f(M),g^N_{f(M)})$ is a surjective conformal submersion onto $f(M)$.
    \item If $f$ is bijective and M is connected, then $f$ is a conformal map.
  \end{enumerate}
  \end{thm}

\begin{proof} From Proposition \ref{rmnp19} and Theorem \ref{rmnt20}, we conclude (1), (2) and  (3). From (1) and (2), we conclude (4).
\end{proof}


%

\section*{References}
\begin{thebibliography}{99}

\bibitem{Fischer}A. Fischer, Riemannian maps between riemannian manifolds. Contemp. Math. .,
132:331-366, 1992.

\bibitem{Gray}
A. Gray, Pseudo-riemannian almost product manifolds and submersions. J. Math.
Mech., 16:715-737, Aug. - Sep., 1957.
\bibitem{Fuglede}
B. Fuglede, Harmonic morphisms between riemannian manifolds. Ann. Inst.
Fourier (Grenoble), 28(2):107-144, 1978.
\bibitem{O'Neill}
B. O'Neill, The fundamental equations of a submersion. Mich. Math. J., (13):
459-469, 1966.

\bibitem{Baired}
Baird, Paul and Wood, John C., Harmonic morphisms between riemannian manifolds. The Clarendon Press, Oxford University Press, Oxford, 2003.

\bibitem{Sahin}
B. Sahin, Conformal riemannian maps between riemannian manifolds, their harmonicity
and decomposition theorems. Acta Appl. Math., 109:829-847, 2010.

\bibitem{Hf-Wt}
G. Golub and C. F. V. Loan, Matrix Computations. The Johns Hopkins University
Press, 1996.

\bibitem{Ddne1}
J. Dieudonne, In Treatise on Analysis and Applications, Volume II. Academic
Press, New York, 1976.
\bibitem{Abraham}
Abraham, R., Marsden, J. E. and  Ratiu, T. Manifolds, Tensor Analysis, and Applications.
Springer-Verlag, New York, 1988.
\bibitem{rbthesis}
R. B. Yadav, Density results in $C(S^n; S^n)$ for lower dimensions and riemannian
morphisms. IIT Guwahati, 2014. Thesis

\end{thebibliography}

\end{document}